\documentclass[a4paper,12pt]{article}

\usepackage {amsfonts, amssymb, amscd,amsthm, amsmath, eucal}

\usepackage{amsbsy}
\usepackage[all]{xy}

\usepackage{color}

\theoremstyle{plain} {
  \newtheorem{thm}{Theorem}[section]
  
  \newtheorem{cor}[thm]{Corollary}
  \newtheorem{lem}[thm]{Lemma}
  \newtheorem{prop}[thm]{Proposition}
  \theoremstyle{definition}
  \newtheorem{rem}[thm]{Remark}
    \newtheorem{constr}[thm]{Construction}
  \theoremstyle{plain}
  \newtheorem{clm}[thm]{Claim}

  \newtheorem{sssection}[thm]{}
}

{

}
\renewcommand{\subsubsection}{\sssection\rm}

\newcommand{\bG}{\mathbf G}

\renewcommand{\P}{\mathbb P}

\DeclareMathOperator{\spec}{Spec}

\newcommand{\can}{\text{\rm can}}

\newcommand{\id}{\text{\rm id}}
\newcommand{\pr}{\text{\rm pr}}

\newcommand{\inc}{\text{\rm inc}}

\newcommand{\const}{\text{\rm const}}
\newcommand{\Spec}{\text{\rm Spec}}

\newcommand{\Aff}{\mathbf {A}}
\newcommand{\Pro}{\mathbf {P}}

\newcommand \xra {\xrightarrow }

\newcommand \hra {\hookrightarrow }

\newcommand{\ttf}{{\text{f}}}

\renewcommand{\P}{\mathbb P}

\newcommand\mydim{\text{\rm dim}}

\renewcommand \id{\operatorname{id}}

\renewcommand \phi\varphi

\newcommand{\et}{\text{\rm\'et}}

\newcommand{\ZZ}{\mathbb Z}

\begin{document}

\title{A theory of nice triples and a theorem due to O.Gabber
}

\author{Ivan Panin\footnote{The author acknowledges support of the
RNF-grant 14-11-00456.}
}


\maketitle

\begin{abstract}
In a series of papers \cite{Pan0}, \cite{Pan1}, \cite{Pan2}, \cite{Pan3} we give a detailed and better structured
proof of the Grothendieck--Serre's conjecture
for semi-local regular rings containing a finite field. The outline of the proof
is the same as in \cite{P1},\cite{P2},\cite{P3}.
If the semi-local regular ring contains an infinite field,
then the conjecture is proved in \cite{FP}. {\it Thus the conjecture
is true for regular local rings containing a field.
}

The present paper is the one \cite{Pan0} in that series.
Theorem \ref{MainHomotopyIntrod} is one of the main result of the paper.
The proof of the latter theorem is completely geometric.
It is based on a theory of nice triples from \cite{PSV}
and on its extension from \cite{P}.
The theory of nice triples is inspired by
the Voevodsky theory of standart triples \cite{V}.

Theorem \ref{MainHomotopyIntrod} yields
{\it an unpublished result} due to O.Gabber
(see the theorem \ref{PSV}=the theorem \ref{Gabbers_Theorem}).

\end{abstract}

\section{Main results}\label{Introduction}
Let $R$ be a commutative unital ring. Recall that an $R$-group scheme $\bG$ is called \emph{reductive},
if it is affine and smooth as an $R$-scheme and if, moreover,
for each algebraically closed field $\Omega$ and for each ring homomorphism $R\to\Omega$ the scalar extension $\bG_\Omega$ is
a connected reductive algebraic group over $\Omega$. This definition of a reductive $R$-group scheme
coincides with~\cite[Exp.~XIX, Definition~2.7]{SGA3}.
A~well-known conjecture due to J.-P.~Serre and A.~Grothendieck
(see~\cite[Remarque, p.31]{Se}, \cite[Remarque 3, p.26-27]{Gr1}, and~\cite[Remarque~1.11.a]{Gr2})
asserts that given a regular local ring $R$ and its field of fractions~$K$ and given a reductive group scheme $\bG$ over $R$, the map
\[
  H^1_{\text{\'et}}(R,\bG)\to H^1_{\text{\'et}}(K,\bG),
\]
induced by the inclusion of $R$ into $K$, {\it has a trivial kernel.}
If $R$ contains an infinite field, then the conjecture is proved in [FP].

For a scheme $U$ we denote by $\mathbb A^1_U$ the affine line over $U$ and by $\P^1_U$ the projective line over $U$.
Let $T$ be a $U$-scheme. By a principal $\bG$-bundle over $T$ we understand a principal $\bG\times_UT$-bundle.
We refer to
\cite[Exp.~XXIV, Sect.~5.3]{SGA3}
for the definitions of
a simple simply-connected group scheme over a scheme
and a semi-simple simply-connected group scheme over a scheme.



\begin{thm}\label{PSV}
Let $k$ be a finite field. Let $\mathcal O$ be the semi-local ring of finitely many closed points on a $k$-smooth irreducible affine $k$-variety $X$
and let $K$ be its field of fractions. Let $\bG$ be a simply-connected reductive group scheme over $k$.
Then
the map
$$\text{H}^1_{et}(\mathcal O,G) \to \text{H}^1_{et}(K,G),$$
induced by the inclusion $\mathcal O$ into $K$, has trivial kernel.
\end{thm}
The latter theorem is
{\it an unpublished theorem} due to O.Gabber.

\begin{thm}
\label{MainHomotopyIntrod}
Let $k$ be a field. Let $\mathcal
O$ be the semi-local ring of finitely many {\bf closed points} on a
$k$-smooth irreducible affine $k$-variety $X$.
Set $U=\spec \mathcal O$.
Let $\bG$ be a reductive
group scheme over $k$.
Let $\mathcal G$ be a principal $\bG$-bundle over $U$
trivial over the generic point of $U$.
Then there
exists a principal $\bG$-bundle $\mathcal G_t$ over the affine line $\mathbb A^1_U=\spec {\mathcal O}[t]$
and a monic polynomial $h(t) \in \mathcal O[t]$ such
that
\par
(i) the $\bG$-bundle $\mathcal G_t$ is trivial over the open subscheme $(\mathbb A^1_U)_h$ in $\mathbb A^1_U$ given by $h(t)\ne0$;
\par
(ii) the restriction of $\mathcal G_t$ to $\{0\}\times U$ coincides
with the original $\bG$-bundle $\mathcal G$.
\par
(iii) $h(1) \in \mathcal O$ is a unit.
\end{thm}
If the field $k$ is infinite a stronger result is proved in
\cite[Thm.1.2]{PSV}.
Theorem \ref{MainHomotopyIntrod} is easily derived
from
Theorem \ref{MajorIntrod} (= \cite[Thm. 1.2]{P}).

\begin{thm}[Geometric]\label{MajorIntrod}
Let $X$ be an affine $k$-smooth irreducible $k$-variety, and let $x_1,x_2,\dots,x_n$ be closed points in $X$.
Let $U=Spec(\mathcal O_{X,\{x_1,x_2,\dots,x_n\}})$ and $\textrm{f}\in k[X]$ be
a non-zero function vanishing at each point $x_i$. Then
there is a monic polinomial $h\in O_{X,\{x_1,x_2,\dots,x_n\}}[t]$,
a commutative diagram
of schemes with the irreducible affine $U$-smooth $Y$
\begin{equation}
\label{SquareDiagram2_2}
    \xymatrix{
       (\Aff^1 \times U)_{h}  \ar[d]_{inc} && Y_h:=Y_{\tau^*(h)} \ar[ll]_{\tau_{h}}  \ar[d]^{inc} \ar[rr]^{(p_X)|_{Y_h}} && X_f  \ar[d]_{inc}   &\\
     (\Aff^1 \times U)  && Y  \ar[ll]_{\tau} \ar[rr]^{p_X} && X                                     &\\
    }
\end{equation}
and a morphism $\delta: U \to Y$ subjecting to the following conditions:
\begin{itemize}
\item[\rm{(i)}]
the left hand side square
is an elementary {\bf distinguished} square in the category of affine $U$-smooth schemes in the sense of
\cite[Defn.3.1.3]{MV};
\item[\rm{(ii)}]
$p_X\circ \delta=can: U \to X$, where $can$ is the canonical morphism;
\item[\rm{(iii)}]
$\tau\circ \delta=i_0: U\to \Aff^1 \times U$ is the zero section
of the projection $pr_U: \Aff^1 \times U \to U$;
\item[\rm{(iv)}] $h(1)\in \mathcal O[t]$ is a unit.
\end{itemize}
\end{thm}

The author thanks A.~Suslin for his interest in the topic of the present article. He
also thanks to A.Stavrova for paying his attention to Poonen's works on Bertini type theorems
for varieties over finite fields. He thanks D.Orlov for useful comments concerning
the weighted projective spaces tacitely involved in the construction of elementary fibrations.
He thanks M.Ojanguren for many inspiring ideas arising from our joint works with him.

\section{Proof of Theorem \ref{MainHomotopyIntrod}}\label{MHIntrod}

\begin{proof}[Proof of Theorem \ref{MainHomotopyIntrod}]
The $U$-group scheme $\bG$ is defined over the base field $k$. We may and will suppose that the principal $\bG$-bundle
$\mathcal G$ is the restriction to $U$ of a principal $\bG$-bundle $\mathcal G'$ on $X$,
and the restriction of $\mathcal G'$ to an principal open subset $X_{\text{f}}$ is trivial.
If $U=\text{Spec}(\mathcal O_{X,x_1,...,x_n})$, then we may and will suppose that
$\text{f}$ vanishes at each point $x_i$.

Theorem \ref{MajorIntrod} (= \cite[Thm. 1.2]{P}) states that there are a monic polinomial
$h\in \mathcal O_{X,x_1,...,x_n}[t]$,
a commutative diagram
(\ref{SquareDiagram2_2})
of schemes with the irreducible affine $U$-smooth $Y$,
and a morphism $\delta: U \to Y$
subjecting to conditions
(i) to (iv) from
Theorem \ref{MajorIntrod}.

Now take the monic polinomial
$h\in \mathcal O_{X,x_1,...,x_n}[t]$
as the desired polinomial
and construct
the desired  principal $\bG$-bundle on $\Aff^1\times U$
as follows.

Take the pull-back $p^*_X(\mathcal G')$ of $\mathcal G'$ to $Y$.
The restriction of $p^*_X(\mathcal G')$ to $Y_h$ is trivial, since
the restriction of $\mathcal G'$ to $X_{\text{f}}$ is trivial.
Take now the trivial $\bG$-bundle over the principal open subset $(\Aff^1\times U)_h$
and glue it with $p^*_X(\mathcal G')$ via an isomorphism over $Y_h$.
This way we get a principal $\bG$-bundle $\mathcal G_t$ over $\Aff^1\times U$.
Clearly, the monic polinomial
$h$ and the principal $\bG$-bundle on $\Aff^1\times U$
are the desired ones.
\end{proof}

\section{Simply-connected case of a theorem due to Gabber}\label{GabbersThm}
An unpublished theorem due to Gabber states particularly that if the base field $k$ is finite, then the Grothendieck--Serre conjecture
is true for any reductive group scheme $\bG$ over $k$. The main aim of the present section is
to recover that result in the simply-connected case.
\begin{thm}\label{Gabbers_Theorem}
Let $k$ be a finite field and let $R$ be a regular local ring containing $k$,
and let $K$ be its field of fractions. Given a simply-connected reductive group scheme $\bG$ over $k$, the map
\[
  H^1_{\text{\'et}}(R,\bG)\to H^1_{\text{\'et}}(K,\bG),
\]
induced by the inclusion of $R$ into $K$, {\it has a trivial kernel.}
\end{thm}

\begin{proof}
The case of a general regular local ring containing $k$ is easily reduced
to the case, when $R$ is the semi-local ring of a finitely many
closed points on an affine $k$-smooth variety $X$.
Moreover we may and will suppose that $\bG$ is
a simple simply-connected $k$-group.

The $k$-group scheme $\bG$ is defined over $k$ and $k$ is finite, and
$\bG$ is
a simple simply-connected $k$-group.
Hence, $\bG$ contains a $k$-Borel subgroup scheme.
Particularly, $\bG$ is isotropic. The bundle $\mathcal G_t$
and the monic polinomial $h$ from
Theorem \ref{MainHomotopyIntrod}
satisfy the hypotheses of Theorem
\cite[Thm.1.3]{PSV}.
Thus the principal $\bG$-bundle
$\mathcal G_t$
from
Theorem \ref{MainHomotopyIntrod}
is trivial. Hence the restriction of $\mathcal G_t$ to $\{0\}\times U$
is trivial. From the other side the latter restriction coincides
with the original $\bG$-bundle $\mathcal G$.
Hence the original $\bG$-bundle $\mathcal G$ is trivial.
\end{proof}

\end{document}